\documentclass[11pt]{amsart}
\usepackage[foot]{amsaddr}
\usepackage[english]{babel}
\usepackage[
pdftex,hyperfootnotes]
{hyperref}
\usepackage[usenames,dvipsnames,svgnames,table,x11names]{xcolor}
\usepackage{pgfplots}
\usepackage{tikz}
\usepackage{tikz-3dplot}
\usetikzlibrary{calc,shadows,shapes.callouts,shapes.geometric,shapes.misc,positioning,patterns,%
	decorations.pathmorphing,decorations.markings,decorations.fractals,decorations.pathreplacing,shadings,fadings}
\pgfplotsset{compat=1.14} 
\usepackage[utf8]{inputenc}
\usepackage{nicematrix,booktabs}

\usepackage[textwidth=17.5cm,textheight=22.275cm,
height=
24.275cm,
width=18cm]{geometry}
\usepackage{comment}
\usepackage{rotating,multirow,pdflscape}
\usepackage{amssymb,latexsym,amsmath,amsthm}
\usepackage{mathrsfs}
\usepackage{mathtools}
\usepackage{bigints}
\usepackage{tensor}

\mathtoolsset{centercolon}
\usepackage[toc,page]{appendix}

\usepackage{tocvsec2}

\usepackage{Baskervaldx}
\newtheorem{theorem}{Theorem}
\newtheorem{pro}{ Proposition }

\newtheorem{rem}{Remark}

\renewcommand{\d}{\operatorname{d}}

\newcommand{\diag}{\operatorname{diag}}

\newcommand{\C}{\mathbb{C}}

\newcommand{\N}{\mathbb{N}}
\newcommand{\R}{\mathbb{R}}

\newcommand{\res}[2]{\operatorname{Res}\left(#1,#2\right)}

\makeatletter
\def\@settitle{\begin{center}%
		\baselineskip14\p@\relax
		\bfseries
		\uppercasenonmath\@title
		\@title
		\ifx\@subtitle\@empty\else
		\\[1ex]\uppercasenonmath\@subtitle
		\footnotesize\mdseries\@subtitle
		\fi
	\end{center}%
}
\def\subtitle#1{\gdef\@subtitle{#1}}
\def\@subtitle{}
\makeatother

\begin{document}
\title[Discrete Pearson Equations, Christoffel and Geronimus transformations]{Pearson Equations for  Discrete Orthogonal Polynomials:\\III.
	Christoffel and Geronimus transformations}

\author[M Mañas]{Manuel Mañas}
\email{manuel.manas@ucm.es}
\address{Departamento de Física Teórica, Universidad Complutense de Madrid, Plaza Ciencias 1, 28040-Madrid, Spain \&
	Instituto de Ciencias Matematicas (ICMAT), Campus de Cantoblanco UAM, 28049-Madrid, Spain}

\thanks{Thanks financial support from the Spanish ``Agencia Estatal de Investigación" research project [PGC2018-096504-B-C33], \emph{Ortogonalidad y Aproximación: Teoría y Aplicaciones en Física Matemática}.}


\begin{abstract}
Contiguous hypergeometric relations  for semiclassical discrete orthogonal polynomials are  described as  Christoffel and Geronimus transformations. Using the Christoffel--Geronimus--Uvarov formulas  quasi-determinatal expressions for the shifted semiclassical discrete orthogonal polynomials are obtained. 

\end{abstract}


\maketitle


\section{Introduction}

Discrete orthogonal polynomials is an important part in the theory of orthogonal polynomials and has many applications.  This is well illustrated  by several reputed monographs on the theme. Let us cite here  \cite{NSU}, devoted to the study of  classical discrete orthogonal polynomials and its applications, and \cite{baik} where the Riemann--Hilbert problem is the key for the study of asymptotics and  further applications of these polynomials. The mentioned relevance of discrete orthogonal polynomials it is also illustrated by  numerous sections or chapters devoted to its discussion in excellent books on orthogonal polynomials such as   \cite{Ismail,Ismail2,Beals_Wong,walter}.
For semiclassical discrete orthogonal polynomials the weight satisfies a  discrete Pearson equation, we refer the reader to \cite{diego_paco,diego_paco1} and \cite{diego,diego1} and references therein for a comprehensive account.  For  the  generalized Charlier and Meixner weights,    Freud--Laguerre type equations for the coefficients of the three term recurrence has been discussed, see for example \cite{clarkson,filipuk_vanassche0,filipuk_vanassche1,filipuk_vanassche2,smet_vanassche}. 

This paper is a sequel of \cite{Manas_Fernandez-Irrisarri}. There we used the Cholesky factorization of the moment matrix to study  discrete orthogonal polynomials $\{P_n(x)\}_{n=0}^\infty$ on the homogeneous lattice, and  studied  semiclassical discrete orthogonal polynomials . The corresponding moments are now given  in terms of   generalized hypergeometric functions.  We constructed a banded semi-infinite matrix  $\Psi$, that we named as Laguerre--Freud structure matrix,  that models the shifts by $\pm 1$  in the independent variable of the sequence  of orthogonal polynomials $\{P_n(x)\}_{n=0}^\infty$.  It was shown that the contiguous relations  for the generalized  hypergeometric functions are symmetries  of the corresponding moment matrix, and that the 3D Nijhoff--Capel  discrete  lattice \cite{nijhoff,Hietarinta}  describes the corresponding contiguous shifts for the squared norms of the orthogonal polynomials. 
In  \cite{Fernandez-Irrisarri_Manas} we considered the generalized Charlier, Meixner and Hahn of type I discrete orthogonal polynomials, and   analyzed the  Laguerre--Freud structure matrix $\Psi$. We  got non linear recurrences for the recursion coefficients of the type
\begin{align*}
	\gamma_{n+1} &=F_1 (n,\gamma_n,\gamma_{n-1},\dots,\beta_n,\beta_{n-1}\dots), &
	\beta_{n+1 }&= F_2 (n,\gamma_{n+1},\gamma_n,\dots,\beta_n,\beta_{n-1},\dots),
\end{align*}
for some functions $F_1,F_2$. Magnus  \cite{magnus,magnus1,magnus2,magnus3} named, attending to \cite{laguerre,freud}, as Laguerre--Freud relations.

In this paper, we return to the hypergeometric contiguous relations and its translation into symmetries of the moment matrix given in \cite{Manas_Fernandez-Irrisarri}, and prove that they are described as  simple Christoffel and Geronimus transformations.
We also show that for these discrete orthogonal polynomials we can find determinantal expressions \emph{à la Christoffel} for the shifted orthogonal polynomials, for that aim we use the general theory of Geronimus--Uvarov perturbations. 

  Christoffel discussed  Gaussian quadrature rules  in  \cite{Christoffel1858uber}, and found explicit formulas relating   sequences of orthogonal polynomials corresponding to two measures  $\d x$ and $p(x) \d x$, with $p(x)=(x-q_1)\cdots(x-q_N)$. The so called Christoffel formula  is a basic result which can be found in a number of orthogonal polynomials textbooks  \cite{Szego1939Orthogonal,Chihara1978introduction,Gautschi2004Orthogonal}.
 Its right inverse is called the Geronimus transformation, i.e., the elementary or canonical Geronimus transformation is a new moment linear functional $\check{u}$  such that $(x-a)\check{u}= u$. In this case  we can write $\check u=(x-a)^{-1}u+\xi\delta(x-a)$, where $\xi\in\R$ is a free parameter and  $\delta(x)$ is the Dirac functional supported at the point $x=a$ \cite{Geronimus1940polynomials}.  We refer to \cite{nuestro0,nuestro1,nuestro2} and references therein for a recent account of the state of the art regarding these transformations.

\subsection{Discrete orthogonal polynomials and discrete Pearson equation}
Let us consider  a measure
$	\rho=\sum_{k=0}^\infty \delta(z-k) w(k)$ with support on $\N_0:=\{0,1,2,\dots\}$,  
for some weight function $w(z)$ with finite values $w(k)$ at the nodes $k\in\N_0$. The corresponding bilinear form is
$\langle F, G\rangle=\sum_{k=0}^\infty F(k)G(k)w(k) $, and the corresponding   moments are given by
\begin{align}\label{eq:moments}
\rho_n=\sum_{k=0}^\infty k^n w(k).
\end{align}
Consequently, the  moment  matrix is
\begin{align*}
	G&=(G_{n,m}), &G_{n,m}&=\rho_{n+m},  & n,m\in\N_0.
\end{align*}
If the moment matrix is such that  all its truncations, which are Hankel matrices, $G_{i+1,j}=G_{i,j+1}$,
\begin{align*}
	G^{[k]}=\begin{pNiceMatrix}
		G_{0,0}&\Cdots &G_{0,k-1}\\
		\Vdots & & \Vdots\\
		G_{k-1,0}&\Cdots & G_{k-1,k-1}
	\end{pNiceMatrix}=\begin{pNiceMatrix}[columns-width = 0.5cm]
		\rho_{0}&\rho_1&\rho_2&\Cdots &\rho_{k-1}\\
		\rho_1 &\rho_2 &&\Iddots& \rho_k\\
		\rho_2&&&&\Vdots\\
		\Vdots& &\Iddots&&\\[9pt]
		\rho_{k-1}&\rho_k&\Cdots &&\rho_{2k-2}
	\end{pNiceMatrix}
\end{align*}
are nonsingular; i.e. the Hankel determinants $\varDelta_k:=\det G^{[k]} $ do not cancel, $\varDelta_k\neq 0 $, $k\in\N_0$, then there exists monic polynomials
\begin{align}\label{eq:polynomials}
	P_n(z)&=z^n+p^1_n z^{n-1}+\dots+p_n^n, & n&\in\N_0,
\end{align}
with $p^1_0=0$, such that the following  orthogonality conditions are fulfilled
\begin{align*}
	\big\langle \rho, P_n(z)z^k\big\rangle &=0, & k&\in\{0,\dots,n-1\},&
	\big\langle \rho, P_n(z)z^n\big\rangle &=H_n\neq 0.
\end{align*}
Moreover, the set $\{P_n(z)\}_{n\in\N_0}$ is an orthogonal set of polynomials
\begin{align*}
	\big\langle\rho,P_n(z)P_m(z)\big\rangle&=\delta_{n,m}H_n,&
	n,m&\in\N_0.
\end{align*}
The second kind functions are given by
\begin{align}\label{eq:second_discrete}
	Q_n(z)= \sum_{k\in\N_0}\frac{P_n(k)w(k)}{z-k}.
\end{align}

In terms of the semi-infinite vector of monomials
\begin{align*}
	\chi(z):=\begin{pNiceMatrix}
		1\\z\\z^2\\\Vdots
	\end{pNiceMatrix}
\end{align*}
we have $G=\left\langle\rho, \chi\chi^\top\right\rangle$, and it becomes evident that the moment matrix  is symmetric, $G=G^\top$.
The vector of monomials $\chi$ is an eigenvector of the \emph{shift matrix}
\begin{align*}
	\Lambda:=\left(\begin{NiceMatrix}[columns-width = auto]
		0 & 1 & 0 &\Cdots&\\
		\Vdots& \Ddots&\Ddots &\Ddots&\\
		&&&&\\
		&&&&
	\end{NiceMatrix}\right)
\end{align*}
i.e., $\Lambda \chi(z)=z\chi(z)$.
From here it follows immediately   that
$\Lambda G=G\Lambda^\top$,
i.e., the Gram matrix is a Hankel matrix, as we previously said.
Being the moment matrix symmetric its Borel--Gauss factorization reduces to a Cholesky factorization
\begin{align}\label{eq:Cholesky}
	G=S^{-1}HS^{-\top}
\end{align}
where $S$ is a lower unitriangular matrix that can be written as 
\begin{align*}
	S=\left(\begin{NiceMatrix}[columns-width = auto]
		1 & 0 &\Cdots &\\
		S_{1,0 } &  1&\Ddots&\\
		S_{2,0} & S_{2,1} & \Ddots &\\
		\Vdots & \Ddots& \Ddots& 
	\end{NiceMatrix}\right)
\end{align*}
and $H=\diag(H_0,H_1,\dots)$ is a  diagonal matrix, with $H_k\neq 0$, for $k\in\N_0$.
The Cholesky  factorization does hold whenever the principal minors of the moment matrix; i.e., the Hankel determinants $\varDelta_k$,  do not cancel.

The components $P_n(z)$ of the  semi-infinite vector of  polynomials
\begin{align}\label{eq:PS}
	P(z):=S\chi(z),
\end{align}
are the monic orthogonal polynomials of the functional $\rho$. From the Cholesky factorization we get
$	\left\langle\rho, \chi\chi^\top\right\rangle=G=S^{-1}HS^{-\top}$  so that
$	\left\langle\rho, \chi\chi^\top\right\rangle S^\top=H$. Therefore,
$\left\langle\rho, S\chi\chi^\top S^\top\right\rangle=H$
and we obtain $	\left\langle\rho, PP^\top \right\rangle=H$,  which encodes the orthogonality of the polynomial sequence $\{P_n(z)\}_{n=0}^\infty$.
The lower Hessenberg matrix
\begin{align}\label{eq:Jacobi}
	J=S\Lambda S^{-1}
\end{align}
that has the vector $P(z)$ as eigenvector with eigenvalue $z$
$	JP(z)=zP(z)$.

The lower Pascal matrix, built up of  binomial numbers, is defined by
\begin{align*}
	B&=(B_{n,m})=\left(\begin{NiceMatrix}[columns-width =auto]
		1&0&\Cdots\\
		1&1&0&\Cdots\\
		1&2&1&0&\Cdots&\\		
		1& 3 & 3&1 & 0&\Cdots\\
		1&4 & 6 & 4 & 1&0&\Cdots\\
		1& 5 & 10 &10 &5&1&0&\Cdots\\
		\Vdots & & & & & &\Ddots&\Ddots
	\end{NiceMatrix}\right), & B_{n,m}&:= \begin{cases}
		\displaystyle \binom{n}{m}, & n\geq m,\\
		0, &n<m.
	\end{cases}
\end{align*}
so that
$	\chi(z+1)=B\chi(z)$. The \emph{dressed Pascal matrices,} are the following lower unitriangular semi-infinite matrices   
\begin{align*}
	\Pi&:=SBS^{-1}, & \Pi^{-1}&:=SB^{-1}S^{-1}, 
\end{align*}
which happen to be  connection matrices; indeed, they satisfy
\begin{align*}
	P(z+1)&=\Pi P(z), & P(z-1)&=\Pi^{-1}P(z).
\end{align*}

The Hankel condition  $\Lambda G=G\Lambda^\top$  and  the Cholesky factorization leads to
$	\Lambda S^{-1} H S^{-\top} =S^{-1} H S^{-\top} \Lambda^\top$,
or, equivalently,
\begin{align}\label{eq:symmetry_J}
	J H=(JH)^\top =HJ^\top.
\end{align}
Hence, $JH$ is symmetric, thus being Hessenberg and symmetric we deduce that $J$  is tridiagonal. Therefore, the Jacobi matrix  \eqref{eq:Jacobi} can be written as follows
\begin{align*}
	J=\left(\begin{NiceMatrix}[columns-width = 0.5cm]
		\beta_0 & 1& 0&\Cdots& \\[-3pt]
		\gamma_1 &\beta_1 & 1 &\Ddots&\\
		0 &\gamma_2 &\beta_2 &1 &\\
		\Vdots&\Ddots& \Ddots& \Ddots&\Ddots 
	\end{NiceMatrix}\right)
\end{align*}
and the eigenvalue relation $JP=zP$ is an order 3 homogeneous linear recurrence  relation 
\begin{align*}
	zP_n(z)&=P_{n+1}(z)+\beta_n P_n(z)+\gamma_n P_{n-1}(z),
\end{align*}
that with the  initial conditions $P_{-1}=0$ and $P_0=1$ completely determine   the set of orthogonal polynomial sequence $\{P_n(z)\}_{n\in\N_0}$ in terms of the recursion coefficients $\beta_n,\gamma_n$.

Given any block matrix 
$	M=\begin{pNiceMatrix}[small]
		A &  B\\
		C & D
	\end{pNiceMatrix}$
with blocks $A\in\C^{r\times r}, B\in\C^{r\times s}, C\in\C^{s\times r}, D\in\C^{s\times s}$, being $A$ a non singular matrix,  we define the Schur complement
$M/A:=D- CA^{-1}B\in\C^{s\times s}$.
When  $s=1$, so that $D\in\C$ and $B,C^\top\in\C^r$ one can show that $M/A\in\C$ is a quotient of determinants
$	M/A=\frac{\det M}{\det A}$.
These Schur complements are the building blocks of the theory of quasi-determinants that we will not treat here. For $s=1$, using  Olver's notation \cite{olver}  for the last quasi determinant 
\begin{align*}
	\Theta_*\begin{pNiceArray}{c|c}
		A&B
		\\
		\midrule
		C & D
	\end{pNiceArray}=D- CA^{-1}B=\frac{\det \begin{psmallmatrix}
			A &  B\\
			C & D
	\end{psmallmatrix}}{\det A}.
\end{align*}

The discrete Pearson equation for the weight is
\begin{align}\label{eq:Pearson0}
\nabla (\sigma w)&=\tau w,
\end{align}
with $\nabla f(z)=f(z)-f(z-1)$, that is $\sigma(k) w(k)-\sigma(k-1) w(k-1)=\tau(k)w(k)$, for $k\in\{1,2,\dots\}$, 
with $\sigma(z),\tau(z)\in\R[z]$ polynomials.
If we write  $\theta:=\tau-\sigma$, the previous Pearson equation reads
\begin{align}\label{eq:Pearson}
\theta(k+1)w(k+1)&=\sigma(k)w(k), &
k\in\N_0.
\end{align}
\begin{theorem}[Hypergeometric symmetry of the moment matrix]
	Let the weight $w$ be subject to a discrete Pearson equation of the type \eqref{eq:Pearson}, where the functions $\theta,\sigma$  are  polynomials, with  $\theta(0)=0$. Then, the corresponding moment matrix fulfills
	\begin{align}\label{eq:Gram symmetry}
	\theta(\Lambda)G=B\sigma(\Lambda)GB^\top.
	\end{align}
\end{theorem}
\begin{rem}
	This result extends to the case when $\theta$ and $\sigma$ are entire functions, not necessarily polynomials, and we can ensure some meaning to $\theta(\Lambda)$ and $\sigma(\Lambda)$. 
\end{rem}

We can use the Cholesky factorization of the Gram matrix \eqref{eq:Cholesky} and the Jacobi matrix \eqref{eq:Jacobi} to get
\begin{pro}[Symmetry of the Jacobi matrix]
	Let the weight $w$ be subject to a discrete Pearson equation of the type \eqref{eq:Pearson}, where the functions $\theta,\sigma$  are entire functions, not necessarily polynomials, with  $\theta(0)=0$.  Then, 
	\begin{align}\label{eq:Jacobi symmetry}
\Pi^{-1}	H\theta(J^\top)=\sigma(J)H\Pi^\top.
	\end{align}	
	Moreover, the matrices $H\theta(J^\top)$ and $\sigma(J)H$ are symmetric.
\end{pro}

In the standard discrete Pearson equation the functions $\theta,\sigma$ are polynomials.
Let us denote their respective degrees by
$N+1:=\deg\theta(z)$ and  $M:=\deg\sigma(z)$.
The roots of these polynomials are denoted by $\{-b_i+1\}_{i=1}^{N}$ and $\{-a_i\}_{i=1}^M$.  Following \cite{diego_paco} we choose 
\begin{align*}
\theta(z)&= z(z+b_1-1)\cdots(z+b_{N}-1), &
\sigma(z)&= \eta (z+a_1)\cdots(z+a_M),
\end{align*}
Notice that we have normalized $\theta$ to be a monic polynomial, while $\sigma$ is not monic, being the coefficient of the leading power denoted by $\eta$. 
Therefore,  the weight is proportional to
\begin{align}\label{eq:Pearson_weight}
w(z)=\frac{(a_1)_z\cdots(a_M)_z}{\Gamma(z+1)(b_1)_z\cdots(b_{N})_z}\eta^z,
\end{align}
see \cite{diego_paco}, where the Pochhammer symbol is understood as
$ (\alpha)_{z}={\frac {\Gamma (\alpha+z)}{\Gamma (\alpha)}}$.

\begin{rem}
	The $0$-th moment  is
\begin{align*}
\rho_0&=\sum_{k=0}^\infty w(k)=\sum_{k=0}^\infty \frac{(a_1)_k\cdots(a_M)_k}{(b_1+1)_k\cdots(b_{N}+1)_k}\frac{\eta^k}{k!}=\tensor[_M]{F}{_{N}} (a_1,\dots,a_M;b_1,\dots,b_{N};\eta)
=
{\displaystyle \,{}_{M}F_{N}\left[{\begin{matrix}a_{1}&\cdots &a_{M}\\b_{1}&\cdots &b_{N}\end{matrix}};\eta\right].}
\end{align*}
is the generalized hypergeometric function, where we are using the two standard 
notations,
see \cite{generalized_hypegeometric_functions}.
Then,  according to
\eqref{eq:moments}, for $n\in\N$, the corresponding  higher moments  $\rho_n=\sum_{k=0}^\infty k^n w(k)$, are
\begin{align*}
\rho_n&=\vartheta_\eta^n\rho_0=\vartheta_\eta^n\Big({\displaystyle \,{}_{M}F_{N}\left[{\begin{matrix}a_{1}&\cdots &a_{M}\\b_{1}&\cdots &b_{N}\end{matrix}};\eta\right]}\Big), &\vartheta_\eta:=\eta\frac{\partial }{\partial \eta}.
\end{align*}
\end{rem}

\begin{theorem}[Laguerre--Freud structure matrix]
	Let us assume that the weight $w$ is subject to the discrete Pearson equation \eqref{eq:Pearson} with  $\theta,\sigma$ polynomials such that $\theta(0)=0$, $\deg\theta(z)=N+1$, $ \deg\sigma(z)=M$. 	Then, the Laguerre--Freud structure matrix
\begin{align}\label{eq:Psi}
\Psi&:=\Pi^{-1}H\theta(J^\top)=\sigma(J)H\Pi^\top=\Pi^{-1}\theta(J)H=H\sigma(J^\top)\Pi^\top\\
&=\theta(J+I)\Pi^{-1} H=H\Pi^\top\sigma(J^\top-I),\label{eq:Psi2}
\end{align}
has only  $N+M+2$ possibly nonzero  diagonals ($N+1$ superdiagonals and $M$ subdiagonals) 
\begin{align*}
\Psi=(\Lambda^\top)^M\psi^{(-M)}+\dots+\Lambda^\top \psi^{(-1)}+\psi^{(0)}+
\psi^{(1)}\Lambda+\dots+\psi^{(N+1)}\Lambda^{N+1},
\end{align*}
for some diagonal matrices $\psi^{(k)}$. In particular,  the lowest subdiagonal and highest superdiagonal  are given by
\begin{align}\label{eq:diagonals_Psi}
\left\{
\begin{aligned}
(\Lambda^\top)^M\psi^{(-M)}&=\eta(J_-)^MH,&
\psi^{(-M)}=\eta H\prod_{k=0}^{M-1}T_-^k\gamma=\eta\diag\Big(H_0\prod_{k=1}^{M}\gamma_k, H_1\prod_{k=2}^{M+1}\gamma_k,\dots\Big),\\
\psi^{(N+1)} \Lambda^{N+1}&=H(J_-^\top)^{N+1},&
\psi^{(N+1)}=H\prod_{k=0}^{N}T_-^k\gamma=\diag\Big(H_0\prod_{k=1}^{N+1}\gamma_k, H_1\prod_{k=2}^{N+2}\gamma_k,\dots\Big).
\end{aligned}
\right.
\end{align}
The vector $P(z)$ of orthogonal polynomials fulfill the following structure equations
\begin{align}\label{eq:P_shift}
\theta(z)P(z-1)&=\Psi H^{-1} P(z), &
\sigma(z)P(z+1)&=\Psi^\top H^{-1} P(z).
\end{align}
\end{theorem}




Three important relations fulfilled by the generalized hypergeometric functions are
\begin{align}
\label{eq:hyper1}\left(\vartheta_\eta+a_{i}\right){}_{M}F_{N}
\left[
\begin{gathered}
	\begin{alignedat}{5}
	a_{1}&&\cdots &&a_{i}&\cdots &&a_{M}
	\end{alignedat}\\[-5pt]
\begin{alignedat}{3}
	b_{1}&&\cdots &&b_{N}
\end{alignedat}
\end{gathered};\eta\right]&=a_{i}\;{}_{M}F_{N}\left[
\begin{gathered}
\begin{alignedat}{5}
a_{1}&&\cdots &&a_{i}+1&\cdots &&a_{M}
\end{alignedat}\\[-5pt]
\begin{alignedat}{3}
b_{1}&&\cdots &&b_{N}
\end{alignedat}
\end{gathered}
;\eta\right],\\
\label{eq:hyper2}
	\left(\vartheta_\eta+b_{j}-1\right){}_{M}F_{N}\left[
		\begin{gathered}
	\begin{alignedat}{3}
	a_{1}&&\cdots &&a_{M}
	\end{alignedat}\\[-5pt]
	\begin{alignedat}{5}
	b_{1}&&\cdots &&b_{j}&\cdots&b_{N}
	\end{alignedat}
	\end{gathered};\eta\right]&=(b_{j}-1)\;{}_{M}F_{N}\left[
	\begin{gathered}
\begin{alignedat}{3}
a_{1}&&\cdots &&a_{M}
\end{alignedat}\\[-5pt]
\begin{alignedat}{5}
b_{1}&&\cdots &&b_{j}-1&\cdots&b_{N}
\end{alignedat}
\end{gathered}
	;\eta\right],\\
	\label{eq:hyper3}
	{\frac {\rm {d}}{{\rm {d}}\eta}}\;{}_{M}F_{N}\left[{\begin{matrix}a_{1}&\cdots &a_{M}\\b_{1}&\cdots &b_{N}\end{matrix}};\eta\right]&=\kappa\;{}_{M}F_{N}\left[{\begin{matrix}a_{1}+1&\cdots &a_{M}+1\\b_{1}+1&\cdots &b_{N}+1\end{matrix}};\eta\right], &\kappa:={\frac {\prod _{i=1}^{M}a_{i}}{\prod _{j=1}^{N}b_{j}}}.
\end{align}
that imply
\begin{align}\label{eq:edo_hyper}
\eta\prod _{n=1}^{M}\left(\eta{\frac {\rm {d}}{{\rm {d}}\eta}}+a_{n}\right)u&=\eta{\frac {\rm {d}}{{\rm {d}}\eta}}\prod _{n=1}^{N}\left(\eta{\frac {\rm {d}}{{\rm {d}}\eta}}+b_{n}-1\right)u, & u:={}_{M}F_{N}\left[{\begin{matrix}a_{1}&\cdots &a_{M}\\b_{1}&\cdots &b_{N}\end{matrix}};\eta\right].
\end{align}
In \eqref{eq:hyper1} and \eqref{eq:hyper2} we have a basic relation between contigous generalized hypergeometric functions and its derivatives.

For the analysis of these equations let us introduce the shift operators  in the parameters
$\{a_i\}_{i=1}^M$ and $\{b_j\}_{j=1}^N$. Thus, given a function $f\left[{\begin{smallmatrix}a_{1}&\cdots &a_{M}\\b_{1}&\cdots &b_{N}\end{smallmatrix}}\right]$ of these parameters we introduce the shifts ${}_i T$ and $T_j$ as follows
\begin{align*}
{}_i T f\left[\begin{gathered}
\begin{alignedat}{5}
a_{1}&&\cdots &&a_{i}&\cdots &&a_{M}
\end{alignedat}\\[-5pt]
\begin{alignedat}{3}
b_{1}&&\cdots &&b_{N}
\end{alignedat}
\end{gathered}\right]&=f\left[\begin{gathered}
\begin{alignedat}{5}
a_{1}&&\cdots &&a_{i}+1&\cdots &&a_{M}
\end{alignedat}\\[-5pt]
\begin{alignedat}{3}
b_{1}&&\cdots &&b_{N}
\end{alignedat}
\end{gathered}\right],&
T_jf\left[	\begin{gathered}
\begin{alignedat}{3}
a_{1}&&\cdots &&a_{M}
\end{alignedat}\\[-5pt]
\begin{alignedat}{5}
b_{1}&&\cdots &&b_{j}&\cdots&b_{N}
\end{alignedat}
\end{gathered}\right]&=f\left[	\begin{gathered}
\begin{alignedat}{3}
a_{1}&&\cdots &&a_{M}
\end{alignedat}\\[-5pt]
\begin{alignedat}{5}
b_{1}&&\cdots &&b_{j}-1&\cdots&b_{N}
\end{alignedat}
\end{gathered}\right],
\end{align*}
and a total shift  $T={}_1T\cdots {}_MT \,T_1^{-1}\cdots T_N^{-1}$; i.e,
\begin{align*}
Tf\left[{\begin{matrix}a_{1}&\cdots &a_{M}\\b_{1}&\cdots &b_{N}\end{matrix}}\right]:=f\left[{\begin{matrix}a_{1}+1&\cdots &a_{M}+1\\b_{1}+1&\cdots &b_{N}+1\end{matrix}}\right].
\end{align*}
Then, we find:
\begin{pro}[Hypergeometric relations]
	The moment matrix $G=(\rho_{n+m})_{n,n\in\N_0}$ satisfies the following hypergeometric relations
\begin{subequations}\label{eq:Gram_hyper}
		\begin{align}
\label{eq:Gram_hyper1}	(\Lambda+a_iI)G&=a_i \;{}_i T G,\\
\label{eq:Gram_hyper2}		(\Lambda+(b_j -1)I)G&=(b_j-1)T_jG,\\
\label{eq:Gram_hyper3}		\Lambda G &=\kappa B (T G) B^\top.
	\end{align}
\end{subequations}
 Finally, from \eqref{eq:edo_hyper} we derive, in an alternative manner, the relation
\eqref{eq:Gram symmetry}.
\end{pro}

\section{A Christoffel--Geronimus  perspective}
The reader familiar with Christoffel and Geronimus transformations probably noticed a remarkable similarity of those transformations with these shifts to  contiguous hypergeometric parameters.  The Pochammer symbol satisfies
\begin{align*}
(\alpha+1)_{z}&={\frac {\Gamma (z+\alpha+1)}{\Gamma (\alpha+1)}}={\frac {(z+\alpha))\Gamma (z+\alpha)}{\alpha\Gamma (\alpha)}}=\frac {z+\alpha}{\alpha}(\alpha)_z,\\
\frac{1}{(\beta-1)_{z}}&=\frac {\Gamma (\beta-1)}{\Gamma (z+\beta-1)}=\frac {(\beta-1+z))\Gamma (\beta)}{(\beta-1)\Gamma (\beta+z)}=\frac {z+\beta-1}{\beta-1}\frac{1}{(\beta)_z}.
\end{align*}
From the explicit form of the weight \eqref{eq:Pearson_weight} 
we get
\begin{align}\label{eq:contigous_christoffel}
\left\{\begin{aligned}
a_i\,({}_jTw)&=(z+a_i)w,& j&\in\{1,\dots,M\},\\
(b_j-1)\,(T_jw)&=(z+b_j-1)w, & j&\in\{1,\dots,N\}.
\end{aligned}\right.
\end{align}
Thus, $a_j\,{}_jT$ and $b_kT_k$ are Christoffel transformations. Moreover, from \eqref{eq:contigous_christoffel} we get
\begin{align}\label{eq:pre_contigous_geronimus}
\left\{\begin{aligned}
(a_i-1)w&=(z+a_i-1)({}_jT^{-1}w), & i&\in\{1,\dots,M\},\\
b_jw&=(z+b_j)(T_j^{-1}w), & j&\in\{1,\dots,N\},
\end{aligned}\right.
\end{align}
so that the inverse transformations are
\begin{align}\label{eq:contigous_geronimus}
\left\{\begin{aligned}
\frac{1}{a_i-1}({}_i T^{-1}w)&=\frac{w}{z+a_i-1}, & i&\in\{1,\dots,M\},\\
\frac{1}{b_j}(T_j^{-1}w)&=\frac{w}{z+b_j}, & j&\in\{1,\dots,N\}.
\end{aligned}\right.
\end{align}
Consequently,  $\frac{1}{a_i-1}{}_iT^{-1}$ and $\frac{1}{b_k}T_k^{-1}$ are massless Geronimus transformations. As is well known, the solutions  to \eqref{eq:pre_contigous_geronimus} are more general than ${}_jT^{-1}w$ and  $T_j^{-1}w$, respectively. In fact, the more general solutions to \eqref{eq:pre_contigous_geronimus} are given by
\begin{align*}
&{}_iT^{-1}w+{}_im \delta(z+a_i-1), &
&T_j^{-1}w+m _j \delta(z+b_j), 
\end{align*}
for some arbitrary constants ${}_jm $ and $m_j $, known as masses, respectively. For  the contiguous transformations discussed here these masses are chosen to cancel.
Finally, for the total shift $T$ we have
\begin{align*}
\kappa Tw(z)=\frac{\prod_{i=1}^M(z+a_i)}{\prod_{j=1}^N(z+b_j)}w(z)
\end{align*}
that for $z=k\in\N_0$, using the Pearson equation \eqref{eq:Pearson}, reads
\begin{align*}
\kappa	Tw(k)=\frac{1}{\eta}(k+1)w(k+1).
\end{align*}
Consequently, we find
\begin{align*}
T^{-1}w(k)&=	(T^{-1}\kappa)	\frac{\eta}{k}w(k-1), &T^{-1}\kappa&=\frac{\prod_{i=1}^M(a_i-1)}{\prod_{j=1}^N (b_j-1)}, &
k\in\N.
\end{align*}

\subsection{The Christoffel contiguous transformations }

In order to apply the Cholesky factorization of the moment matrix to the previous result we introduce the following  semi-infinite matrices
\begin{subequations}
	\begin{align}
{}_i\omega&:=({}_jTS)(\Lambda+a_iI) S^{-1}, & {}_i\Omega&:=S\,({}_iTS)^{-1}, & i&\in\{1,\dots,M\},
\label{jomega}\\
\omega_j&:=(T_kS)(\Lambda+(b_j -1)I) S^{-1}, &\Omega_k&:=S\,(T_jS)^{-1},& j&\in\{1,\dots,N\},\label{omegak}\\
\omega&:=(TS)B^{-1}\Lambda S^{-1}, & \Omega&:=SB(TS)^{-1}, \label{eq:omegas}
\end{align}
\end{subequations}
that, as we immediately show, are connection matrices.
The action of these matrices on the vector of orthogonal polynomials lead to the following:

\begin{pro}[Connection formulas] \label{pro:connection_Christoffel}
	The following relations among orthogonal polynomials are satisfied
\begin{subequations}\label{eq:prueba_subs}
		\begin{align}\label{eq:connection1}
{}_i\omega P(z)&=(z+a_i) {}_iT P(z), & {}_i\Omega\;\;{}_iTP(z)&=P(z), & i&\in\{1,\dots,M\},\\\label{eq:connection2}
\omega_jP(z)&=(z+b_j-1) T_jP(z), 
&\Omega_j \,T_jP(z)&=P(z),& j&\in\{1,\dots,N\},\\\label{eq:connection3}
\omega  P(z)&=(z-1) \;TP(z-1), & \Omega \;TP(z)&=P(z+1).
\end{align}
\end{subequations}
\end{pro}

The Cholesky factorization of the Gram matrices leads to the following expressions for these connection matrices:

\begin{pro}	Let us assume that the Cholesky factorization of the Gram matrices $G,{}_j TG,T_kG$ and $TG$ hold. Then, we have the following expressions 
	\begin{align*}
\begin{aligned}
{}_i\omega&=
\left(\begin{NiceMatrix}[columns-width = 0.7cm,small]
	\frac{a_i \,{}_iTH_0}{H_0} &1 &0 &\Cdots\\
	0& \frac{a_i \,{}_iTH_1}{H_1} &1&\Ddots\\
	\Vdots & \Ddots &\Ddots & \Ddots
\end{NiceMatrix}\right),&
\omega_j&=
\left(\begin{NiceMatrix}[columns-width = 0.7cm,small]
\frac{(b_j-1)T_jH_0}{H_0} &1 &0 &\Cdots\\
	0&	\frac{(b_j-1) T_jH_1}{H_1} &1&\Ddots\\
	\Vdots & \Ddots &\Ddots & \Ddots
\end{NiceMatrix}\right),
&
\omega&=\left(\begin{NiceMatrix}[columns-width = 0.7cm,small]
\kappa\frac{TH_0}{ H_0} &1 &0 &\Cdots\\
0&  \kappa\frac{TH_1}{ H_1} &1&\Ddots\\
	\Vdots & \Ddots &\Ddots & \Ddots
\end{NiceMatrix}\right),\\[5pt]
{}_i\Omega&=
\left(\begin{NiceMatrix}[columns-width = 0.5cm,small]
1 &0 &\Cdots&\\
\frac{1}{a_i} \frac{H_1}{{}_iTH_0} & 1& \Ddots\\
 0 &\frac{1}{a_i} \frac{H_2}{{}_iTH_1} & \Ddots&\\
\Vdots &\Ddots &\Ddots & 
\end{NiceMatrix}\right),&
\Omega_j&=
\left(\begin{NiceMatrix}[columns-width = 0.5cm,small]
1 &0 &\Cdots&\\
\frac{1}{b_j-1} \frac{H_1}{T_jH_0} & 1& \Ddots\\
0 &\frac{1}{b_j-1} \frac{H_2}{T_jH_1} & \Ddots&\\
\Vdots &\Ddots &\Ddots & 
\end{NiceMatrix}\right),&
\Omega&=
\left(\begin{NiceMatrix}[columns-width = 0.5cm,small]
1 &0 &\Cdots&\\
\frac{H_1}{\kappa TH_0} & 1& \Ddots\\
0 &\frac{H_2}{\kappa TH_1} & \Ddots&\\
\Vdots &\Ddots &\Ddots & 
\end{NiceMatrix}\right).
\end{aligned}
\end{align*}
\end{pro}\vspace*{5pt}
\begin{proof}
	In the one hand, observe  that ${}_j\omega$, $\omega_k$ and $\omega$ are lower uni-Hessenberg matrices, i.e. all its superdiagonals are zero but for the first one that is $\Lambda$, while in the other hand ${}_j\Omega$, $\Omega_k$ and $\Omega$
	are lower unitriangular matrices.
From \eqref{eq:Gram_hyper} 	
we get
	\begin{align*}
	(\Lambda+a_iI)S^{-1} H S^{-\top}&=a_i \;({}_j T S)^{-1} ({}_iTH) ({}_i TS)^{-\top}\\
	(\Lambda+(b_j -1)I)S^{-1} H S^{-\top}&=(b_j-1)(T_jS)^{-1} (T_jH) (T_jS)^{-\top},\\
	B^{-1}\Lambda S^{-1} H S^{-\top} &={\frac {\prod _{i=1}^{M}a_{i}}{\prod _{j=1}^{N}b_{j}}}  (T S^{-1}) (TH)(TS)^{-\top}B^\top
	\end{align*}
	that can be written as follows
\begin{subequations}\label{eq:Omega}
		\begin{align}
\label{eq:Omega1}	{}_i\omega H&=a_i  ({}_i TH) ({}_i\Omega)^{\top}\\
	\omega_j H &=(b_j-1) (T_jH) (\Omega_j)^{\top},\label{eq:Omega2}\\
	\omega H &={\frac {\prod _{i=1}^{M}a_{i}}{\prod _{j=1}^{N}b_{j}}}  (TH)\Omega^\top.
\label{eq:Omega3}
	\end{align}
\end{subequations}
		From these relations given that ${}_j\omega,\omega_k $ and $\omega$ are lower uni-Hessenberg matrices and $({}_j\Omega)^\top,(\Omega_k)^\top$ and $\Omega$ are upper unitriangular matrices, we conclude that ${}_j\omega,\omega_k$  and  $\omega$ are upper triangular matrices with only the main diagonal and the first superdiagonal non vanishing and that ${}_j\Omega,\Omega_k$ and $\Omega$ are lower unitriangular matrices with only the first subdiagonal different from zero.
The given expressions follow by identification of the coefficients in \eqref{eq:Omega}.
\end{proof}

Let $\mathscr Z=\cup_{n\in\N_0}\mathscr Z_n$, with $\mathscr Z_n$ being the set of zeros $P_n$.

\begin{theorem}[Christoffel formulas]
Whenever, $\big(\{-a_i\}_{i=1}^M\cup \{-b_j+1\}\cup \{1\}_{j=1}^N\big)\cap \mathscr Z=\varnothing$,  the following expressions are fulfilled
\begin{align*}
{}_iTP_n(z)&=\frac{1}{z+a_i}\Big(
P_{n+1}(z)-\frac{P_{n+1}(-a_i)}{P_{n}(-a_i)} P_n(z)\Big),& i&\in\{1,\dots,M\},\\
T_jP_n(z)&=\frac{1}{z+b_j-1}\Big(P_{n+1}(z)-\frac{P_{n+1}(-b_j+1)}{P_{n}(-b_j+1)}P_n(z)\Big),&j&\in\{1,\dots,N\},\\
TP_n(z-1)&=\frac{1}{z-1}\Big(P_{n+1}(z)-\frac{P_{n+1}(1)}{P_{n}(1)} P_n(z)\Big).
\end{align*}
\end{theorem}
\begin{proof}
From the connection formulas 	we obtain
\begin{subequations}
	\begin{align}\label{eq:connection1_zero}
{}_i\omega P(-a_i)&=0,  & i&\in\{1,\dots,M\},\\\label{eq:connection2_zero}
\omega_j P(-b_j+1)&=0, &j&\in\{1,\dots,N\},\\\label{eq:connection3_zero}
\omega  P(1)&=0.
\end{align}
\end{subequations}
so that
\begin{subequations}
	\begin{align}\label{eq:connection1_zero_1}
a_i\frac{{}_iTH_n}{H_n} &=-\frac{P_{n+1}(-a_i)}{P_{n}(-a_i)},  & i&\in\{1,\dots,M\},\\\label{eq:connection2_zero_1}
(b_j-1)\frac{T_kH_n}{H_n} &=-\frac{P_{n+1}(-b_j+1)}{P_{n}(-b_j+1)}, &j&\in\{1,\dots,N\},\\\label{eq:connection3_zero_1}
{\frac {\prod _{i=1}^{M}a_{i}}{\prod _{j=1}^{N}b_{j}}} \frac{TH_n}{H_n}&=-\frac{P_{n+1}(1)}{P_{n}(1)}.
\end{align}
\end{subequations}
From the connection formulas we get the result.
\end{proof}
\begin{theorem}[Jacobi matrix and $LU$ and $UL$ factorization]\label{theorem:JacobiLU}
The following LU factorizations hold true
	\begin{align}\label{eq:LUJ}
\left\{\begin{aligned}
	J+a_iI&={}_iL \;{}_iU, &  {}_iL&:= {}_i\Omega, & {}_iU&:=a_i({}_iTH) {}_i\Omega^{\top}H^{-1},& i&\in\{1,\dots,M\},\\
	J+(b_j-1)I&=L_j U_j, & L_j&:=\Omega_j, &U_j&=(b_j-1)(T_jH) \Omega_j^{\top}H^{-1},&j&\in\{1,\dots,N\},\\
	J  &=  LU, & L&:=\Omega, &U&:= {\frac {\prod _{i=1}^{N}a_{i}}{\prod _{j=1}^{M}b_{j}}}(TH)\Omega^{\top}H^{-1}. 
\end{aligned}\right.
	\end{align}
	Moreover,  the Christoffel transformed Jacobi matrices have the following $UL$ factorizations
\begin{subequations}\label{eq:ULJ}
		\begin{align}\label{eq:ULJ1}
	{}_iTJ+a_i I&={}_i U \; {}_i L,& i&\in\{1,\dots,M\},\\\label{eq:ULJ2}
	T_jJ+(b_j-1)I&=U_jL_j,&j&\in\{1,\dots,N\},\\\label{eq:ULJ3}
	TJ-I&=UL.
	\end{align}
\end{subequations}
\end{theorem}
\begin{proof}
	From \eqref{eq:Gram_hyper}
	we get
\begin{align*}
S(\Lambda+a_iI)S^{-1}  &=a_i \; S({}_i T S)^{-1} ({}_iTH) ({}_i TS)^{-\top}S^{\top}H^{-1},\\
S(\Lambda+(b_j -1)I)S^{-1} &=(b_j-1)S(T_jS)^{-1} (T_jH) (T_jS)^{-\top}S^{\top}H^{-1},\\
S\Lambda S^{-1} &={\frac {\prod _{i=1}^{N}a_{i}}{\prod _{j=1}^{M}b_{j}}}  SB(T S^{-1}) (TH)(TS)^{-\top}B^\top S^{\top}H^{-1}. 
\end{align*}
from where \eqref{eq:LUJ} follow.
To prove \eqref{eq:ULJ} we write  \eqref{eq:Gram_hyper1} and 	\eqref{eq:Gram_hyper2} 
	\begin{align*}
{}_jiTS	(\Lambda+a_jI)({}_iTS	)^{-1}&=a_i ({}_i TH)({}_i TS)^{-\top}S^\top H^{-1}S ({}_j TS	)^{-1}, \\
T_jS	(\Lambda+(b_j -1)I)(T_j S)^{-1}&=(b_j-1)(T_jH) (T_jS)^{-\top}S^\top H^{-1}S (T_jS)^{-1},
\end{align*}
and we get \eqref{eq:ULJ1} and \eqref{eq:ULJ2}. To show \eqref{eq:ULJ3} we write
 \eqref{eq:Gram_hyper3} as follows
	\begin{align*}
B^{-1}\Lambda S^{-1}H
&={\frac {\prod _{i=1}^{M}a_{i}}{\prod _{j=1}^{N}b_{j}}}   (T S)^{-1} (TH) (TS)^{-\top} B^\top S^\top ,
\end{align*}
and recalling that $B^{-1}\Lambda=(\Lambda-I)B^{-1}$
we obtain
	\begin{align*}
(TS)(\Lambda - I)(TS)^{-1} (TS)B^{-1} S^{-1}H
&={\frac {\prod _{i=1}^{M}a_{i}}{\prod _{j=1}^{N}b_{j}}}   (TH) (TS)^{-\top} B^\top S^\top .
\end{align*}
That is, we deduce that
\begin{align*}
(TJ-I)\Omega^{-1}={\frac {\prod _{i=1}^{M}a_{i}}{\prod _{j=1}^{N}b_{j}}}  (TH)\Omega^\top H^{-1},
\end{align*}
and the third $UL$ factorization follows.
\end{proof}

\begin{rem}
	Given a symmetric tridiagonal matrix
\begin{align*}
\mathscr J=\left(\begin{NiceMatrix}
r_0& s_0& 0& 0&\Cdots\\
s_0 & r_1 & s_1 &0&\Ddots\\
0 &s_1 &r_2&s_2 &\Ddots\\
\Vdots&\Ddots&\Ddots&\Ddots&\Ddots
\end{NiceMatrix}\right)\end{align*}
its Cholesky factorization is
\begin{align*}
\mathscr J&=L D L^\top, &
L&=\left(\begin{NiceMatrix}[columns-width =auto]
1 &0 &0&\Cdots\\
l_1& 1&0 &\Ddots\\
0& l_2&1&\Ddots\\
\Vdots&\Ddots&\Ddots&\Ddots
\end{NiceMatrix}\right),& D&=\diag(\delta_0,\delta_1,\dots),
\end{align*}
with $\delta_0=r_0$, $l_1=\frac{s_0}{\delta_0}$ and
\begin{align*}
\delta_n&=r_n-\frac{s_{n-1}^2}{\delta_{n-1}}, &
l_{n+1}&=\frac{s_n}{\delta_n}, & n\in\N.
\end{align*}
Which, when iterated leads to continued fraction expressions for the Cholesky  factor's coefficients in terms of the $\{r_n,s_n\}_{n\in\N_0}$.
Equating $\mathscr J$ with $(J+a_j I)H$, $(J+(b_k-1)I)H$ and $JH$ (which are symmetric tridiagonal  matrices) and applying the above formulas we get expressions for $({}_j\Omega,{}_jTH)$,and $(\Omega_k,T_kH)$) and $(\Omega, TH)$, respectively. The coefficients $(r_n,s_n)$ are $(\beta_nH_n+a_i,H_{n+1})$,
 $(\beta_nH_n+b_k-1,H_{n+1})$ and  $(\beta_nH_n,H_{n+1})$, respectively. Therefore, we get continued fraction expressions for the $\Omega$'s, $TH$'s and $\omega$'s in terms of the recursion coefficients.
 
\end{rem}
 \subsection{The Geronimus contiguous transformations}

From Proposition \ref{pro:connection_Christoffel} we get the following connections formulas 
\begin{align*}
	({}_iT^{-1}{}_i\omega) \,{}_iT^{-1}P(z)&=(z+a_i-1)  P(z), & ({}_iT^{-1}{}_i\Omega)\;\;P(z)&:={}_iT^{-1}P(z), & i&\in\{1,\dots,M\},\\
	(T_j^{-1}\omega_j )\,T_j^{-1}P(z)&=(z+b_j) P(z), 
	&(T_j^{-1}\Omega_j )P(z)&=T_j^{-1}P(z),& j&\in\{1,\dots,N\},\\
(	T^{-1}\omega)  T^{-1}P(z)&=(z-1) P(z-1), & (T^{-1}\Omega) P(z)&=T^{-1}P(z+1).
\end{align*}
From these connections formulas we do not get Christoffel type formulas as for the Christoffel transformations. We need use  associated second kind functions, see  \eqref{eq:second_discrete}.
\begin{pro}
 For the second kind functions $Q_n(z)$, the following relations hold
\begin{subequations}
		\begin{align}\label{eq:Geronimus_contigous_1}
	(a_i-1)({}_iT^{-1}{_j}\Omega )Q(z)&=(z+a_i-1)	({}_iT^{-1}Q(z))-\begin{pNiceMatrix}
		{}_iT^{-1} H_0
		\\
		0\\
		\Vdots
	\end{pNiceMatrix},& i&\in\{1,\dots,M\},\\\label{eq:Geronimus_contigous_2}
b_j(T_j^{-1}\Omega_j)Q(z)&=	(z+b_j)	(T_j^{-1}Q(z))-\begin{pNiceMatrix}
	T_j^{-1} H_0
	\\
	0\\
	\Vdots
\end{pNiceMatrix},& j&\in\{1,\dots,N\},\\\label{eq:Geronimus_contigous_3}
 \hspace*{-1cm}(T^{-1}\Omega)(\Upsilon Q(z-1)- P(z-1))&=zT^{-1}Q(z)-T^{-1}P(z)-\begin{pNiceMatrix}
	T^{-1} H_0
	\\
	0\\
	\Vdots
	\end{pNiceMatrix},
	\end{align}
\end{subequations}
with $\Upsilon:=\eta\frac{\prod_{i=1}^M(a_i-1)}{\prod_{j=1}^N(b_j-1)}=
\eta T^{-1}\kappa$
\end{pro}

\begin{proof}
	Let us compute 
\begin{align*}\hspace*{-.5cm}
(z+a_i-1)	({}_iT^{-1}Q(z))-(a_i-1)({}_iT^{-1}{_i}\Omega )Q(z)&=\begin{multlined}[t][0.55\textwidth]	
(z+a_i-1)\sum_{k=0}^\infty\frac{	({}_iT^{-1}P(k))({}_iT^{-1}w(k))}{z-k}\\-
\sum_{k=0}^\infty\frac{{}_jT^{-1}P(k)}{z-k} ({}_jT^{-1}w(k))(k+a_j-1)
\end{multlined}\\&=
\sum_{k=0}^\infty({}_iT^{-1}P(k)) ({}_iT^{-1}w(k))=\begin{pNiceMatrix}
	{}_iT^{-1} H_0
	\\
	0\\
	\Vdots
\end{pNiceMatrix}.
\end{align*}
Analogously,
\begin{align*}
	(z+b_j)	(T_j^{-1}Q(z))-b_j(T_j^{-1}\Omega_j )Q(z)&=\begin{multlined}[t][0.65\textwidth]	
		(z+b_j)\sum_{k=0}^\infty\frac{	(T_j^{-1}P(k))(T_j^{-1}w(k))}{z-k}-
		\sum_{k=0}^\infty\frac{T_j^{-1}P(k)}{z-k} (T_j^{-1}w(k))(k+b_j)
	\end{multlined}\\&=
	\sum_{k=0}^\infty(T_j^{-1}P(k)) (T_j^{-1}w(k))=\begin{pNiceMatrix}
		T_j^{-1} H_0
		\\
		0\\
		\Vdots
	\end{pNiceMatrix}.
\end{align*}
Finally, we prove the last equation.
In the one hand, we have
$	T^{-1}Q(z)=\sum_{k=0}^\infty 
	(T^{-1}P(k))\frac{T^{-1}w(k)}{z-k}$.
On the other hand, we find
\begin{align*}
		(T^{-1}\Omega)\Upsilon Q(z-1)&=	(T^{-1}\Omega)\sum_{k=0}^\infty 
	P(k)\frac{\Upsilon w(k)}{z-1-k}=	(T^{-1}\Omega)\sum_{k=1}^\infty 
	P(k-1)\frac{\Upsilon w(k-1)}{z-k}\\
	&=\sum_{k=1}^\infty(T^{-1}P(k))\frac{kT^{-1}w(k)}{z-k}=\sum_{k=0}^\infty(T^{-1}P(k))\frac{kT^{-1}w(k)}{z-k}\\&=
	\sum_{k=0}^\infty(T^{-1}P(k))\big(\frac{z}{z-k}-1\big)T^{-1}w(k)\\&=
	z\sum_{k=0}^\infty(T^{-1}P(k))\frac{T^{-1}w(k)}{z-k}-
	\sum_{k=0}^\infty(T^{-1}P(k))T^{-1}w(k),
\end{align*}
so that
\begin{align}\label{eq:Geronimus_contigous_3.0}
(T^{-1}\Omega)\Upsilon Q(z-1)=T^{-1}Q(z)-\begin{pNiceMatrix}
		T^{-1} H_0
		\\
		0\\
		\Vdots
		\end{pNiceMatrix},
	\end{align}
and using $(T^{-1}\Omega) P(z-1)=T^{-1}P(z)$ we get the announced result.
\end{proof}
 
 Observe that, as far $-a_i+1,-b_j\not\in\N_0$, the discrete  support of $\rho_z$, from \eqref{eq:Geronimus_contigous_1} and \eqref{eq:Geronimus_contigous_2} we obtain
 	\begin{align*}
 	(a_j-1)({}_jT^{-1}{_j}\Omega )Q(-a_j+1)&=-\begin{pNiceMatrix}
 		{}_jT^{-1} H_0
 		\\
 		0\\
 		\Vdots
 	\end{pNiceMatrix},&
 	b_j(T_j^{-1}\Omega_j )Q(-b_j)=-\begin{pNiceMatrix}
 		T_j^{-1} H_0
 		\\
 		0\\
 		\Vdots
 	\end{pNiceMatrix},
 \end{align*}
so that
\begin{align*}
	({}_iT^{-1}{_i}\Omega )_{n,n-1}&=-\frac{Q_n(-a_i+1)}{Q_{n-1}(-a_i+1)}, & n&>1, & 	{}_iT^{-1} H_0&=-(a_i-1)Q_0(-a_i+1) \\
	(T_j^{-1}\Omega_j)_{n,n-1}&=-\frac{Q_n(-b_j)}{Q_{n-1}(-b_j)}, & n&>1,& 	T_j^{-1} H_0&=-b_jQ_0(-b_j).
\end{align*}

Why we write \eqref{eq:Geronimus_contigous_3} instead of the equivalent equation \eqref{eq:Geronimus_contigous_3.0}? Because \eqref{eq:Geronimus_contigous_3} is prepared for the limit $z\to 0$. Notice that $z=0$ belongs to the support $\N_0$ of $\rho_z$, and $\lim_{z\to 0} zT^{-1}Q(z)$ does not necessarily vanishes.  Observe that $T^{-1} Q(z)$ is meromorphic with simple poles at $\N_0$, in fact 
\begin{align*}
	\res{zT^{-1}Q(z)}{0}=T^{-1}P(0)T^{-1}w(0)=T^{-1}P(0)=(T^{-1}\Omega)P(-1),
\end{align*}
where we have used that $w(0)=1$ does not depend on the parameters $a_i,b_j$ and, consequently, $T^{-1}w(0)=1$. 
Hence,
$\lim_{z\to 0} (zT^{-1}Q(z)-T^{-1}P(z))=0$.
Therefore, from \eqref{eq:Geronimus_contigous_3} we obtain that 
\begin{align*}
 (T^{-1}\Omega)(\Upsilon Q(-1)- P(-1))&=-\begin{pNiceMatrix}
	T^{-1} H_0
	\\
	0\\
	\Vdots
\end{pNiceMatrix}
\end{align*}
and, consequently, we deduce 
\begin{align*}
	(T^{-1}\Omega)_{n,n-1}&=\frac{\Upsilon Q_n(-1)-P_n(-1)}{\Upsilon Q_{n-1}(-1)-P_{n-1}(-1)},
	& T^{-1}H_0&=P_n(-1)-\Upsilon Q_0(-1).
\end{align*}

\begin{theorem}
	 For $n\in\N_0$,  the Geronimus transformed orthogonal polynomials we have the Christoffel--Geronimus expressions
 \begin{align*}
 	{}_iT^{-1}P_n(z)&=P_n(z)-\frac{Q_n(-a_i+1)}{Q_{n-1}(-a_i+1)}P_{n-1}(z), & i&\in\{1,\dots,M\},\\
 	T_j^{-1}P_n(z)&=P_n(z)-\frac{Q_n(-b_j)}{Q_{n-1}(-b_j)}P_{n-1}(z),& j&\in\{1,\dots,N\},\\
 	T^{-1}P_n(z)&=P_n(z-1)-\frac{\Upsilon Q_n(-1)-P_n(-1)}{\Upsilon Q_{n-1}(-1)-P_{n-1}(-1)}P_{n-1}(z-1).
 \end{align*}
\end{theorem}

From Theorem \ref{theorem:JacobiLU} we get 
\begin{theorem}[Jacobi matrix and $UL$ and $LU$ factorization]
The Jacobi matrix has following $UL$ factorizations 
		\begin{align*}
		\left\{\begin{aligned}
			J+a_i I&={}_i U \; {}_i L,&  {}_iL&:= {}_iT^{-1}{}_i\Omega, & {}_iU&:=a_iH ({}_iT^{-1}{}_i\Omega)^{\top}({}_iT^{-1}H)^{-1},& i&\in\{1,\dots,M\}\\
			J+(b_j-1)I&=U_jL_j,& L_j&:=T_j^{-1}\Omega_j, &U_j&=(b_j-1)H (T_j^{-1}\Omega_j)^{\top}(T_j^{-1}H)^{-1},&j&\in\{1,\dots,N\},\\
			J-I&=UL,& L&:=T^{-1}\Omega, &U&:= {\frac {\prod _{i=1}^{N}a_{i}}{\prod _{j=1}^{M}b_{j}}}H(T^{-1}\Omega)^{\top}(T^{-1}H)^{-1}. .
		\end{aligned}\right.
	\end{align*}
The Geronimus transformed Jacobi matrices have the following $LU$ factorizations
	\begin{align*}
		\left\{\begin{aligned}
			{}_iT^{-1}J+a_iI&={}_iL \;{}_iU, & i&\in\{1,\dots,M\},\\
		T_j^{-1}J+(b_j-1)I&=L_j U_j, &j&\in\{1,\dots,N\},\\
			T^{-1}J  &=  LU.
		\end{aligned}\right.
	\end{align*}
\end{theorem}

\subsection{Christoffel--Geronimus--Uvarov transformation and shifts in $z$}
Here we follow  \cite{nuestro0,nuestro1,nuestro2} adapted to the scalar case.
If we denote $P^{(\pm)}_n(z)=P_n(z\pm 1)$, we notice that $\{P_n^{(\pm )}(z)\}_{n=0}^\infty$ is a sequence of monic orthogonal polynomials
\begin{align*}
\sum_{k=\mp 1}^\infty P_n^{(\pm )}(k)P_m^{(\pm)}(k)w^{(\pm)}(k)=\delta_{n,m}H_n,
\end{align*}
with $w^{(\pm)}(k):=w(k\pm 1)$.
The two perturbed functionals
$\rho^{(\pm)}:=\sum_{k=\mp 1}^\infty \delta(z-k)w^{(\pm)}(z)$ 
satisfy
\begin{align}\label{eq:Uvarov_correspondece}
\theta(z+1)\rho^{(+)}&=\sigma(z)\rho,&
\sigma(z-1)\rho^{(-)}&=\theta(z)\rho.
\end{align}
Indeed, using the Pearson equation \eqref{eq:Pearson} and that $\theta(0)=0$ we get
\begin{align*}
\theta(z+1)\rho^{(+)}&=\sum_{k=-1}^\infty \delta(z-k)
\theta(z+1)w(z+1)=\sum_{k=0}^\infty \delta(z-k)
\sigma(z)w(z)=\sigma(z)\rho,\\
\sigma(z-1)\rho^{(-)}&=\sum_{k=1}^\infty \delta(z-k)
\sigma(z-1)w(z-1)=\sum_{k=0}^\infty \delta(z-k)
\theta(z)w(z)=\theta(z)\rho.
\end{align*}
 Consequently, the Pearson equation could be understood as describing a perturbation of the functional, a perturbation of Geronimus--Uvarov  type (a composition of a Geronimus and a Christoffel perturbation). If fact, for the $\rho^{(+)}$ perturbation, if $\sigma=1$ we have a Geronimus transformation and for $\theta=1$ we have a Christoffel transformation. The reserve occurs for the $\rho^{(-)}$ perturbation, if $\theta=1$ we have a Geronimus transformation and for $\sigma=1$ we have a Christoffel transformation. 
These interpretations, together with \eqref{eq:P_shift}, allows  to find explicit expressions for the shifted polynomials in terms of Christoffel type formulas that involve the evaluation of the polynomials and the second kind functions at the zeros of $\sigma$ and $\theta$.  

Attending to \eqref{eq:Uvarov_correspondece} and following \cite{nuestro0,nuestro1,nuestro2}  adapted to the scalar case,
we have the  interpretation
\begin{align*}
W_G^{(+)}&=\theta(z+1), & W_C^{(+)} &=\sigma(z),&
W_G^{(-)}&=\sigma(z-1), & W_C^{(-)}&=\theta(z).
\end{align*}
The corresponding perturbed Gram matrices are
\begin{align*}
G^{(\pm)}&=\langle\rho^{(\pm)},\chi\chi^\top\rangle=\sum_{k=\mp 1}^\infty
\chi(k)\chi(k)^\top w^{(\pm)}(k) =\sum_{k=\mp 1}^\infty
\chi(k)\chi(k)^\top w(k\pm 1) \\&=\sum_{k=0}^\infty
\chi(k\mp 1)\chi(k\mp 1)^\top w(k)=B^{\mp 1}\Big(\sum_{k=0}^\infty
\chi(k)\chi(k)^\top w(k) \Big)B^{\mp \top}=B^{\mp 1}GB^{\mp \top}.
\end{align*}

We have 
\begin{align*}
\rho^{(\pm)}&=\sum_{k=-\mp 1}^\infty \delta(z-k)w^{(\pm)}(k)=\sum_{k=-\mp 1}^\infty \delta(z-k)w(k\pm 1),
\end{align*}
and also, using Pearson equation \eqref{eq:Pearson}
\begin{align*}
\frac{\sigma(z)}{\theta(z+1)}\rho&
=\sum_{k=0}^\infty\delta(z-k)\frac{\sigma(k)}{\theta(k+1)}w(k)
=\sum_{k=0}^\infty\delta(z-k)w(k+1),
	\\
	\frac{\theta(z)}{\sigma(z-1)}\rho&
	=\sum_{k=0}^\infty\delta(z-k)	\frac{\theta(k)}{\sigma(k-1)}w(k)=\sum_{k=1}^\infty\delta(z-k)w(k-1).	
\end{align*}
Consequently, we can write
\begin{align*}
\rho^{(+)}&=\frac{\sigma(z)}{\theta(z+1)}\rho+\delta(z+1)w(0),&
\rho^{(-)}&=	\frac{\theta(z)}{\sigma(z-1)}\rho.
\end{align*}
Hence, for the $(+)$ perturbation we need a Geronimus mass  $\delta(z+1)w(0)$, while for the $(-)$ perturbation there is no mass at all. 

The Cholesky factorizations for the corresponding perturbed Gram matrices $G^{(\pm)}$ gives
\begin{align*}
G^{(\pm)}=\big(S^{(\pm)}\big)^{-1}H^{(\pm)}\big(S^{(\pm)}\big)^{-\top}=B^{\mp 1}S^{-1}H S^{-\top}B^{\mp \top},
\end{align*}
and from the uniqueness of such factorization we get
$S^{(\pm)}=S B^{\pm 1}=\Pi^{\pm 1}S$ and  $H^{(\pm)}=H$.
The resolvent matrices, see Definition 2 in \cite{nuestro2},  of these two Geronimus--Uvarov perturbations are
\begin{align*}
\omega^{(\pm)}=S^{(\pm)}W^{(\pm)}_C(\Lambda)S^{-1}=H^{(\pm)}(S^{(\pm)})^{-\top}
W_G^{(\pm)}(\Lambda^\top)S^\top H^{-1}.
\end{align*}
That is,
\begin{align*}
\omega^{(\pm)}&=S B^{\pm 1}W^{(\pm)}_C(\Lambda)S^{-1}=S B^{\pm 1}SS^{-1}W^{(\pm)}_C(\Lambda)S^{-1}=\Pi^{\pm 1}W^{(\pm)}_C(J)\\
&=H^{(\pm)}
\Big(SW_G^{(\pm)}(\Lambda) \big(S^{(\pm)}\big)^{-1}\Big)^\top H^{-1}= H\big(
SW_G^{(\pm)}(\Lambda) B^{\mp 1}S^{-1})^\top H^{-1}=H\big(W_G^{(\pm)}(J)\Pi^{\mp 1} \big)^\top H^{-1}.
\end{align*}
Hence, recalling  iii) in  \cite[Proposition 3 ]{nuestro2}, formulas (5) and (6) we get
\begin{align*}
\omega^{(+)}&=\Pi\sigma(J)=H\Pi^{-\top} \theta(J^\top+I) H^{-1}=H \theta(J^\top)\Pi^{-\top} H^{-1}=\Psi^\top H^{-1},\\
\omega^{(-)}&=\Pi^{-1}\theta(J)=H\Pi^{\top} \sigma(J^\top-I) H^{-1}=H \sigma(J^\top)\Pi^{\top} H^{-1}=\Psi H^{-1}.
\end{align*}
Consequently, we have
\begin{align*}
\sigma(z)P(z+1)&=\omega^{(+)}P(z)=\Pi\sigma(J)P(z), \\
( \omega^{(+)})^\top H^{-1}P(z+1)
&=H^{-1}\theta(J+I)\Pi^{-1}HH^{-1}P(z+1)=\theta(z+1)H^{-1}P(z),\\
\theta(z)P(z-1)&=\omega^{(-)}P(z)=\Pi^{-1}\theta(J)P(z),\\
( \omega^{(-)})^\top H^{-1}P(z-1)
&=H^{-1}\sigma(J-I)\Pi HH^{-1}P(z-1)=\sigma(z-1)H^{-1}P(z).
\end{align*}
%
These equations recover \eqref{eq:P_shift} from this perturbation perspective. More interesting are the results in \cite{nuestro2} regarding Geronimus--Uvarov perturbations and the second kind functions.  
The new perturbed second kind functions are
\begin{align*}
Q^{(\pm)}(z)&=\Big\langle\rho^{(\pm)}_\zeta,\frac{P^{(\pm)}(\zeta)}{z-\zeta}\Big\rangle=\sum_{k=\mp 1}^\infty\frac{P^{(\pm)}(k)w^{(\pm)}(k)}{z-k}=\sum_{k=\mp 1}^\infty\frac{P(k\pm 1)w(k\pm 1)}{z-k}\\
&=\sum_{k=\mp 1}^\infty\frac{P(k\pm 1)w(k\pm 1)}{z\pm 1-(k\pm 1)}=
\sum_{k=0}^\infty\frac{P(k)w(k)}{z\pm 1-k}\\
&=Q(z\pm 1).
\end{align*}
According to the Proof of   \cite[Proposition 4]{nuestro2} we have
\begin{align*}
Q^{(\pm)}(z)W^{(\pm)}_G(z)-\omega^{(\pm)}Q=\left\langle \rho_\zeta^{(\pm)},
P(\zeta)\frac{W_G^{(\pm)}(z)-W_G^{(\pm)}(\zeta)}{z-\zeta}
\right\rangle,
\end{align*}
and we get the following relations
\begin{align*}
Q(z+1)\theta(z+1)-\Psi^\top H^{-1} Q(z)&=\left\langle \rho_\zeta^{(+)},
P(\zeta)\frac{\theta(z+1)-\theta(\zeta+1)}{z-\zeta}\right\rangle\
=\sum_{k=0}^\infty P(k)\frac{\theta(z+1)-\theta(k)}{z+1-k}w(k),
\\
Q(z-1)\sigma(z-1)-\Psi H^{-1} Q(z)&=\left\langle \rho_\zeta^{(-)},
P(\zeta)\frac{\sigma(z-1)-\sigma(\zeta-1)}{z-\zeta}\right\rangle
=\sum_{k=0}^\infty
P(k)\frac{\sigma(z-1)-\sigma(k)}{z-1-k}w(k).
\end{align*}
Finally, we collect these results together.
\begin{pro}
	The following holds
		\begin{align*}
	\theta(z)Q(z)-\Psi^\top H^{-1} Q(z-1)&=\sum_{k=0}^\infty
	P(k)\frac{\theta(z)-\theta(k)}{z-k}w(k),\\
\sigma(z)	Q(z)-\Psi H^{-1} Q(z+1)
	&=\sum_{k=0}^\infty
	P(k)\frac{\sigma(z)-\sigma(k)}{z-k}w(k).
	\end{align*}

\end{pro}

If $\theta(z)=z^{N+1}+\theta_Nz^N+\cdots+\theta_1z$ and $\sigma(z)=\eta z^M+ \sigma_{M-1} z^{M-1}+\dots+\sigma_0$, 	
we have for each of the polynomials in the Pearson equation
\begin{align*}
\frac{\theta(z)-\theta(k)}{z-k}&= 	(	\chi(k))^\top\begin{pmatrix}
M_\theta\chi^{[N+1]}(z)\\0
\end{pmatrix},&
\frac{\sigma(z)-\sigma(k)}{z-k}&=(	\chi(k))^\top\begin{pmatrix}
M_\sigma\chi^{[M]}(z)\\0
\end{pmatrix}.
\end{align*}
where we have used   the matrices 
\begin{align*}\hspace*{-.5cm}
M_\theta&=\begin{pNiceMatrix}
	0& \theta_1 & \Cdots & &\theta_{N} &1\\
	\theta_1  && &\Iddots&\Iddots&0\\
	\Vdots &&&\Iddots&\Iddots&\Vdots\\
	%
	\\
	\theta_{N} &&&&&\\[4pt]
	1&0& \Cdots& & &0
\end{pNiceMatrix}\in\mathbb C^{(N+1)\times (N+1)},&
M_\sigma&=\begin{pNiceMatrix}
	\sigma_0& \sigma_1 & \Cdots & &\sigma_{M-1} &\eta\\
	\sigma_1  && &\Iddots&\Iddots&0\\
	\Vdots &&&\Iddots&\Iddots&\Vdots\\
	%
	\\
	\sigma_{M-1} &&&&&\\[15pt]
\eta&0& \Cdots& & &0
\end{pNiceMatrix}\in\mathbb C^{(M)\times (M)}.
\end{align*}

Therefore,
\begin{align*}
\sum_{k=0}^\infty
P(k)\frac{\theta(z)-\theta(k)}{z-k}w(k)&=S\sum_{k=0}^\infty
\chi(k)	(	\chi(k))^\top w(k)\begin{pmatrix}
M_\theta\chi^{[N+1]}(z)\\0
\end{pmatrix}=SG\begin{pmatrix}
M_\theta(\chi^{[N+1]}(z)\\0
\end{pmatrix}\\&=HS^{-\top}\begin{pmatrix}
M_\theta\chi^{[N+1]}(z)\\0
\end{pmatrix}=\begin{pmatrix}
(HS^{-\top})^{[N+1]}M_\theta\chi^{[N+1]}(z)\\0
\end{pmatrix}\\&=
\begin{pmatrix}
H^{[N]}\big(S^{[N+1]}\big)^{-\top}M_\theta\chi^{[N+1]}(z)\\0
\end{pmatrix}
\end{align*}
So that, the previous Proposition may be recast as follows
\begin{pro}
The following relations are satisfied
		\begin{align*}
\theta(z)Q(z)-\Psi^\top H^{-1} Q(z-1)&=\begin{pmatrix}
H^{[N+1]}\big(S^{[N+1]}\big)^{-\top}M_\theta\chi^{[N+1]}(z)\\0
\end{pmatrix},\\
\sigma(z)Q(z)-\Psi H^{-1} Q(z+1)
&=\begin{pmatrix}
H^{[M]}\big(S^{[M]}\big)^{-\top}M_\sigma\chi^{[M]}(z)\\0
\end{pmatrix},
\end{align*}
and, in particular, we have
\begin{align*}
\theta(z)Q_n(z)&= \sum_{m=n-N-1}^{n+M}\frac{Q_m(z-1)}{H_m}\Psi_{m,n}, & n&>N+1,&
\sigma(z) Q_n(z)&= \sum_{m=n-M}^{n+N+1}\Psi_{n,m}\frac{Q_m(z+1)}{H_m}, &n&>M.
\end{align*}
\end{pro}

From \eqref{eq:second_discrete_2}, if $k\in\N_0$ is not a zero of $P_n$, we see that $Q_n(z)$ is a meromorphic function with simples poles located at $z\in\N_0$, with residues at these poles given by
$\res{Q_n}{k}=P_n(k)w(k)$.

Thus, we get
\begin{align*}
\theta(k) P_n(k) w(k)&= (1-\delta_{k,0})\sum_{m=n-N-1}^{n+M}\frac{P_m(k-1)w(k-1)}{H_m}\Psi_{m,n}, & n&>N,\\
\sigma(k) P_n(k) w(k)&= \sum_{m=n-M}^{n+N+1}\Psi_{n,m}\frac{P_m(k+1)w(k+1)}{H_m}, &n&>M.
\end{align*}
which are in  disguise \eqref{eq:P_shift} evaluated at $k\in\N_0$, i.e.
\begin{align*}
\theta(k)P_n(k-1)&=\sum_{m=n-M}^{n+N}\Psi_{n,m} \frac{ P_m(k)}{H_m}, &
\sigma(k)P_n(k+1)&=\sum_{m=n-N}^{n+M}\frac{P_m(k)}{H_m}\Psi_{m,n} .
\end{align*}

Finally, we have 
\begin{theorem}
	Assume that the zeros $\theta$ and $\sigma$ are simple, so that
\begin{align*}
\theta(z)&= z\prod_{k=1}^{N}(z+b_k-1),&
\sigma(z)&=\eta \prod_{k=1}^{M}(z+a_k),
\end{align*}
with  $b$'s all  different and $a$'s all different. Then, in terms of quasi-determinants (in this case quotients of determinants), for $n\geq M$
{\small\begin{multline*}\hspace*{-.7cm}
\theta(z)P_n(z-1)\\\hspace*{-.7cm}=\Theta_*\begin{bmatrix}
P_{n-M}(0) & P_{n-M}(- b_1+1) & \cdots &P_{n-M}(- b_{N}+1)&Q_{n-M}(- a_1+1)&\cdots &Q_{n-M}(- a_M+1)&P_{n-M}(z)\\\vdots&\vdots & & \vdots& \vdots&&\vdots &\vdots\\
P_{n+N+1}(0) & P_{n+N}(- b_1+1) & \cdots &P_{n+N+1}(- b_{N}+1)&Q_{n+N+1}(- a_1+1)&\cdots &Q_{n+N+1}(- a_M+1)&P_{n+N+1}(z)
\end{bmatrix}
\end{multline*}}
and , for $n\geq N+1$
{\begin{multline*}\hspace*{-.7cm}
\frac{\sigma(z)}{\eta}P_n(z+1)\\\hspace*{-.8cm}=\small\Theta_*\begin{bmatrix}
P_{n-N-1}(- a_1) &  \cdots &P_{n-N-1}(- a_{M})&Q_{n-N-1}(-1)-P_{n-N-1}(-1)&Q_{n-N-1}(- b_1)&\cdots &Q_{n-N-1}(- b_{N})&P_{n-N}(z)\\\vdots& & \vdots& \vdots& \vdots&&\vdots&\vdots \\
P_{n+M}(- a_1) & \cdots &P_{n+M}(- a_{M})&Q_{n+M}(-1)-P_{n+M}(-1)&Q_{n+M}(- b_1)&\cdots &Q_{n+M}(- b_{N})&P_{n+M}(z)
\end{bmatrix}.
\end{multline*}
}
\end{theorem}

\section*{Conclusions and outlook}

Adler and van Moerbeke have throughly used the Gauss--Borel factorization of the moment matrix in their studies of integrable systems and orthogonal polynomials  \cite{adler_moerbeke_1,adler_moerbeke_2,adler_moerbeke_4}. Our Madrid group extended and applied it  in different contexts, namely CMV orthogonal polynomials, matrix orthogonal polynomials, multiple orthogonal polynomials  and multivariate orthogonal, see  \cite{am,afm,nuestro0,nuestro1,nuestro2,ariznabarreta_manas0,ariznabarreta_manas1,ariznabarreta_manas2,ariznabarreta_manas_toledano}.
For a general overview see \cite{intro}. 

Recently \cite{Manas_Fernandez-Irrisarri}  we extended those ideas to the discrete scenario, and study  the consequences of the Pearson equation on the moment matrix and Jacobi matrices. For that description a new banded matrix is required, the Laguerre--Freud structure matrix that encodes the Laguerre--Freud relations for  the recurrence coefficients. We have also found that the contiguous relations fulfilled generalized hypergeometric functions determining the moments of the weight described for the squared norms of the orthogonal polynomials a discrete Toda hierarchy known as Nijhoff--Capel equation, see \cite{nijhoff}. In \cite{Fernandez-Irrisarri_Manas}  these ideas are applied  to generalized Charlier, Meixner, and Hahn orthogonal polynomials extending the results of \cite{diego,smet_vanassche, filipuk_vanassche0,filipuk_vanassche1,filipuk_vanassche2}.

 In this paper we have seen how the contiguous relations could be understood as Christoffel and Geronimus transformations. Moreover, we also us the Geronimus--Uvarov transformations to give determinantal expressions for the shifted discrete orthogonal polynomials.

For the future, we will  study the  generalized Hahn of type II polynomials, and extend these techniques to multiple discrete orthogonal polynomials  \cite{Arvesu_Coussment_Coussment_VanAssche} and its relations with the transformations presented in \cite{bfm}
and  quadrilateral lattices \cite{quadrilateral1,quadrilateral2},  




\begin{thebibliography}{99}
	
	
	
	\bibitem{adler_moerbeke_1} Mark  Adler and  Pierre van Moerbeke,
\emph{	Vertex operator solutions to the discrete KP hierarchy},
	Communications in Mathematical Physics \textbf{203} (1999) 185-210
	
	\bibitem{adler_moerbeke_2} ----------,
\emph{Generalized orthogonal polynomials, discrete KP and Riemann–Hilbert problems},
	Communications in Mathematical Physics \textbf{207} (1999) 589-620.
	
	
	\bibitem{adler_moerbeke_4} ----------,
\emph{	Darboux transforms on band matrices, weights and associated polynomials},
	International  Mathematics Research Notices \textbf{18} (2001) 935-984.
	
	\bibitem{afm}
	Carlos Álvarez-Fernández, Ulises Fidalgo Prieto, and Manuel Mañas,
	\emph{Multiple orthogonal polynomials of mixed type: Gauss-Borel factorization and the multi-component 2D Toda hierarchy},
	Advances in Mathematics~\textbf{227} (2011) 1451–1525.
	
	
	\bibitem{am} 	Carlos Álvarez-Fernández and  Manuel Mañas, \emph{Orthogonal Laurent polynomials on the unit circle, extended CMV ordering and 2D Toda type integrable hierarchies}, Advances in Mathematics \textbf{240} (2013) 132-193
	
\bibitem{nuestro0} Carlos Álvarez-Fernández, Gerardo Ariznabarreta, Juan C. García-Ardila, Manuel Mañas,	and Francisco Marcellán, \emph{Christoffel transformations for matrix orthogonal polynomials in the real line and the non-Abelian 2D Toda lattice hierarchy}, International Mathematics Research Notices \textbf{2017} n\textsuperscript{o}5 (2017) 1285-1341, \hyperref{https://doi.org/10.1093/imrn/rnw027}{}{}{DOI:10.1093/imrn/rnw027}.

\bibitem{nuestro1} Gerardo Ariznabarreta, Juan C. García-Ardila, Manuel Mañas, and Francisco Marcellán, 
\emph{Matrix biorthogonal polynomials on the real line: Geronimus transformations},
Bulletin of Mathematical Sciences \textbf{9} (2019) 195007 (68 pages)
\hyperref{https://doi.org/10.1142/S1664360719500073}{}{}{DOI:10.1142/S1664360719500073 } (World Scientic) or \hyperref{https://doi.org/10.1007/s13373-018-0128-y}{}{}{DOI:10.1007/s13373-018-0128-y3} (Springer-Verlag).
	
	\bibitem{nuestro2} Gerardo Ariznabarreta, Juan C. García-Ardila, Manuel Mañas, and Francisco Marcellán, \emph{Non-Abelian integrable hierarchies: matrix biorthogonal polynomials and perturbations} 
	Journal of  Physics A: Mathematical \&  Theoretical \textbf{51} (2018) 205204.
	
	\bibitem{ariznabarreta_manas0}
	Gerardo Ariznabarreta and Manuel Mañas,
	Matrix orthogonal Laurent polynomials on the unit circle and Toda type integrable systems,
	Advances in Mathematics \textbf{264} (2014) 396-463.
	
	\bibitem{ariznabarreta_manas1}----------, 
	\emph{Multivariate orthogonal polynomials and integrable systems},
	Advances in Mathematics \textbf{302} (2016) 628–739.
	
	\bibitem{ariznabarreta_manas2}
	----------, \emph{Christoffel transformations for multivariate orthogonal polynomials},
	Journal of Approximation Theory \textbf{225} (2018) 242–283.
	
	\bibitem{ariznabarreta_manas_toledano}
	Gerardo Ariznabarreta, Manuel Mañas, and Alfredo Toledano, 
	\emph{CMV Biorthogonal Laurent Polynomials: Perturbations and Christoffel Formulas},
	Studies in Applied Mathematics \textbf{140} (2018) 333–400.
	

\bibitem{Arvesu_Coussment_Coussment_VanAssche} Jorge Arvesú, Jonathan Coussement, and Walter Van Assche, \emph{Some discrete multiple orthogonal polynomials}, Journal of Computational and Applied Mathematics \textbf{153} (2003).

\bibitem{generalized_hypegeometric_functions} Richard A. Askey and Adri B. Olde Daalhuis, \emph{Generalized hypergeometric function} (2010), in Olver, Frank W. J.; Lozier, Daniel M.; Boisvert, Ronald F.; Clark, Charles W. (eds.), NIST Handbook of Mathematical Functions, Cambridge University Press.


\bibitem{baik} Jinho Baik, Thomas Kriecherbauer, Kenneth T.-R. McLaughlin, and Peter D. Miller, \emph{Annals of Mathematics Studies} \textbf{164}, Princeton University Press, 2007.

\bibitem{Beals_Wong} Richard Beals and Roderick Wong, \emph{Special functions and orthogonal polynomials}, Cambridge Studies in Advanced Mathematics \textbf{153}, Cambridge University Press, 2016.



	\bibitem{Chihara1978introduction} Theodore S. Chihara, \emph{ An Introduction to Orthogonal Polynomials}, Dover, 2011. Originally published by Gordon and Breach, 1978.

	\bibitem{Christoffel1858uber}
Elwin B. Christoffel, \emph{ Über die gaußische quadratur und eine
	verallgemeinerung derselben}, Journal für die reine und angewandte Mathematike (Crelle's Journal) \textbf{55} (1858) 61--82.

\bibitem{bfm}
Amílcar Branquinho, Ana Foulquié-Moreno, and Manuel Mañas, 
\emph{Multiple orthogonal polynomials on the step-line}, \hyperref{https://arxiv.org/abs/2106.12707}{}{}{\texttt{arXiv:2106.12707 [CA]}} (2021).

\bibitem{clarkson} Peter A. Clarkson, \emph{Recurrence coefficients for discrete orthonormal polynomials and the Painlevé equations}, Journal of  Physics A: Mathematical \&  Theoretical \textbf{46} (2013) 185205.

\bibitem{quadrilateral1} Adam Doliwa, Paolo Maria Santini, and Manuel Mañas, \emph{Transformations of quadrilateral lattices}, Journal of Mathematical Physics \textbf{41} (2000) 944--990.

\bibitem{diego}Diego Dominici, \emph{Laguerre–Freud equations for generalized Hahn polynomials of type I}, Journal of Difference Equations and Applications \textbf{24} (2018) 916–940.

\bibitem{diego1}	----------,  \emph{Matrix factorizations and orthogonal polynomials}, Random Matrices Theory Applications \textbf{9} (2020) 2040003, 33 pp.

\bibitem{diego_paco} Diego Dominici and Francisco Marcellán, \emph{Discrete semiclassical orthogonal polynomials of class one}, Pacific Journal of Mathematics \textbf{268} n\textsuperscript{o}2 (2012) 389-411.

\bibitem{diego_paco1}   	----------, \emph{Discrete semiclassical orthogonal polynomials of class 2} in 
\emph{Orthogonal Polynomials: Current Trends and Applications}, edited by E. Huertas and F. Marcellán,  SEMA SIMAI Springer Series, \textbf{22} (2021) 103-169, Springer.

\bibitem{Fernandez-Irrisarri_Manas}  Itsaso Fernández-Irrisarri and Manuel Mañas, \emph{Pearson Equations for  Discrete Orthogonal Polynomials: II. Generalized Hypergeometric Functions and Toda Equations}, \hyperref{https://arxiv.org/abs/2107.02177}{}{}{\texttt{arXiv:2107.02177 [CA]}} (2021). 



\bibitem{filipuk_vanassche0} Galina Filipuk and Walter Van Assche,
\emph{Recurrence coefficients of generalized Charlier polynomials and the fifth Painlevé equation},  Proceedings of  American  Mathematical  Society  \textbf{141}  (2013) 551–62.

\bibitem{filipuk_vanassche1} 	----------,
\emph{Recurrence Coefficients of a New Generalization
	of the Meixner Polynomials}, Symmetry, Integrability and Geometry: Methods and Applications (SIGMA) \textbf{7} (2011), 068, 11 pages.

\bibitem{filipuk_vanassche2} 	----------, 
\emph{ Discrete Orthogonal Polynomials with Hypergeometric Weights and Painlevé VI}, Symmetry, Integrability and Geometry: Methods and Applications (SIGMA) \textbf{14} (2018), 088, 19 pages.

\bibitem{freud} Géza Freud. \emph{On the coefficients in the recursion formulae of orthogonal polynomials}, Proceedings of the  Royal Irish Academy Section A \textbf{76} n\textsuperscript{o}1 (1976) 1–6.

		\bibitem{Gautschi2004Orthogonal}
Walter Gautschi, \emph{Orthogonal Polynomials: computation and approximation},
Oxford University Press, 2004.

	\bibitem{Geronimus1940polynomials}
 Yakov L. Geronimus, \emph{ On polynomials orthogonal with regard to a given sequence of numbers and a theorem by W. Hahn}, Izvestiya Akademii Nauk SSSR \textbf{ 4}(1940), 215-228.

\bibitem{Hietarinta} Jarmo Hietarinta, Nalini Joshi and Frank W. Nijhoff, \emph{Discrete Systems and Integrabilty}, Cambridge Texts in Applied Mathematics, Cambridge University Press, 2016.


\bibitem{Ismail} Mourad E. H.Ismail, \emph{Classical and Quantum Orthogonal Polynomails in One Variable}, Encyclopedia of Mathematics and its Applications \textbf{98}, Cambridge University Press, 2009.

\bibitem{Ismail2}
Mourad E. H. Ismail and Walter Van Assche,
\emph{ Encyclopedia of Special Functions: The Askey--Bateman Project. Volume I: Univariate Orthogonal Polynomials},
Edited by Mourad Ismail, Cambridge University Press, 2020.


\bibitem{laguerre} Edmond Laguerre, \emph{Sur la réduction en fractions continues d'une fraction qui satisfait à une  équation différentialle linéaire du premier ordre dont les coefficients sont rationnels. } Journal de Mathématiques Pures et Appliquées  4\textsuperscript{e} série, tome \textbf{1} (1885) 135–165 .

\bibitem{magnus} Alphonse P. Magnus, \emph{ A proof of Freud’s conjecture about the orthogonal polynomials related to $|x|\rho\exp(-x^{2m})$, for integer $m$}, in “Orthogonal polynomials and applications (Bar-le-Duc, 1984)”,  Lecture Notes in Mathematics \textbf{1171} 362–372, Springer, 1985.

\bibitem{magnus1} ----------, \emph{On Freud’s equations for exponential weights}, Journal of Approximation Theory \textbf{46}(1) (1986) 65–99.

\bibitem{magnus2} ----------, \emph{Painlevé-type differential equations for the recurrence coefficients of semi-classical orthogonal polynomials},  Journal of Computational and Applied Mathematics \textbf{57}   (1995) 215–237.

\bibitem{magnus3} ----------, \emph{Freud’s equations for orthogonal polynomials as discrete Painlevé equations}, in “Symmetries and integrability of difference equations (Canterbury, 1996)”, London Mathematical Society Lecture Note Series \textbf{255} 228–243, Cambridge University Press, 1999.

\bibitem{intro} Manuel Mañas, \emph{Revisiting Biorthogonal Polynomials. An LU factorization discussion} in 
\emph{Orthogonal Polynomials: Current Trends and Applications}, edited by E. Huertas and F. Marcellán,  SEMA SIMAI Springer Series, \textbf{22} (2021) 273-308, Springer.

\bibitem{quadrilateral2} Manuel Mañas, Adam Doliwa, and Paolo Maria Santini, \emph{Darboux transformations for multidimensional quadrilateral lattices. I}, Physics Letters A \textbf{232} (1997) 99--105.

\bibitem{Manas_Fernandez-Irrisarri} Manuel Mañas, Itsaso Fernández-Irrisarri, and Omar González-Fernández, \emph{Pearson Equations for  Discrete Orthogonal Polynomials: I. Generalized Hypergeometric Functions and Toda Equations}, \hyperref{https://arxiv.org/abs/2107.01747}{}{}{\texttt{arXiv:2107.01747 [CA]}} (2021).

\bibitem{nijhoff} Frank W. Nijhoff and Hans W. Capel,\emph{ The direct linearisation approach to hierarchies of integrable PDEs in 2 + 1 dimensions: I. Lattice equations and the differential-difference hierarchies}. Inverse Problems \textbf{6} (1990) 567-590. 

\bibitem{NSU} Arthur F. Nikiforov, Sergei K. Suslov, and Vasilii B. Uvarov, \emph{Classical Orhogonal Polynomials of a Discrete Variable},  Springer Series in Computational Physics, Springer, 1991.

%

\bibitem{olver} Peter J. Olver, \emph{On multivariate interpolation}, Studies in  Applied Mathematics \textbf{116} (2006) 201–240.



\bibitem{smet_vanassche} Christophe Smet and Walter Van Assche, \emph{Orthogonal polynomials on a bi-lattice}, Constructive Approximation \textbf{36} (2012) 215–242.

	\bibitem{Szego1939Orthogonal}
Gabor Szegő, \emph {Orthogonal Polynomials},  American
Mathematical Society Colloquium Publications \textbf{23}, American Mathematical Society, 1939. Reprinted 2003.


\bibitem{walter} Walter Van Assche, \emph{Orthogonal Polynomials and Painlevé Equations}, Australian Mathematical Society Lecture Series \textbf{27}, Cambridge University Press, 2018.


\end{thebibliography}
\end{document}